\documentclass[10pt,reqno]{amsart}
\usepackage{amsmath}
\usepackage{amssymb}
\usepackage{amsthm}
\usepackage{eepic,epic}
\usepackage{epsfig}
\usepackage{graphicx}
\newlength{\defbaselineskip}
\newcommand{\setlinespacing}[1]%
           {\setlength{\baselineskip}{#1 \defbaselineskip}}

\numberwithin{equation}{section}

\newtheorem{thm}{Theorem}[section]
\newtheorem{cor}[thm]{Corollary}
\newtheorem{lem}[thm]{Lemma}
\newtheorem{prop}[thm]{Proposition}

\theoremstyle{definition}

\theoremstyle{remark}
\newtheorem{rem}[thm]{Remark}
\numberwithin{equation}{section}

\begin{document}

\title[Integrability of the wave propagator]
{On the integrability of the wave propagator arising from the Liouville-von Neumann equation}

\author{Youngwoo Koh, Yoonjung Lee and Ihyeok Seo}

\thanks{Y. Koh was supported by NRF-2019R1F1A1054310.
I. Seo was supported by NRF-2019R1F1A1061316.}


\address{Department of Mathematics Education, Kongju National University, Kongju 32588, Republic of Korea}
\email{ywkoh@kongju.ac.kr}

\address{Department of Mathematics, Sungkyunkwan University, Suwon 16419, Republic of Korea}
\email{yjglee@skku.edu}

\address{Department of Mathematics, Sungkyunkwan University, Suwon 16419, Republic of Korea}
\email{ihseo@skku.edu}

\begin{abstract}
The Liouville-von Neumann equation describes the change in the density matrix with time.
Interestingly, this equation was recently regarded as a wave equation for wave functions but not a equation for density functions.
This setting leads to an extended form of the Schr\"odinger wave equation governing the motion of a quantum particle.
In this paper we obtain the integrability of the wave propagator arising from the Liouville-von Neumann equation in this setting.
\end{abstract}

\maketitle

\section{Introduction}
The pure state of a quantum system is described by a wave function obeying the Schr\"odinger equation $i\partial_t\psi+\Delta\psi=0$.
For the description of mixed states, the notion of density matrix was introduced by von Neumann \cite{N} (see also \cite{L}).
This density matrix for a pure state is equal to the product of the wave function and its complex conjugate at different arguments.
Motivated by this, one can consider
\begin{align*}
i\frac{\partial}{\partial t}(\psi(x,t)\overline{\psi}(y,t))&=i\frac{\partial\psi(x,t)}{\partial t}\overline{\psi}(y,t)
+i\frac{\partial\overline{\psi}(y,t)}{\partial t}\psi(x,t)\\
&=-\Delta_x\psi(x,t)\overline{\psi}(y,t)+\Delta_y\overline{\psi}(y,t)\psi(x,t)\\
&=(-\Delta_x+\Delta_y)(\psi(x,t)\overline{\psi}(y,t)),
\end{align*}
which leads to the following equation
\begin{equation}\label{nle0}
i\partial_t\Psi(x,y,t)+(\Delta_x-\Delta_y)\Psi(x,y,t)=0
\end{equation}
where $(x,y,t)\in \mathbb{R}^n\times \mathbb{R}^n\times\mathbb{R}$.

This equation is commonly called the Liouville-von Neumann equation in physical literature,
giving the change in the density matrix with time (\cite{N2,L2}).
However, it is more interesting to regard \eqref{nle0} as a wave equation for wave functions but not a equation for density functions.
This is the key point in a recent work by Chen \cite{C}.
Contrary to Schr\"odinger's wave functions,
the wave functions $\Psi(x,y,t)$ of the Liouville-von Neumann equation \eqref{nle0} for a single particle are \textit{bipartite}.
These \textit{bipartite} wave functions satisfy all the basic properties of Schr\"odinger's wave functions which correspond to those
\textit{bipartite} wave functions of product forms.
Indeed, the Schr\"odinger equation is a special case of the equation \eqref{nle0} with the initial data of product form,
$\Psi(x,y,0)=\psi(x)\overline{\psi}(y)$,
because in this case $\Psi(x,y,t)=\psi(x,t)\overline{\psi}(y,t)$ with $\psi(x,t)$ satisfying the Schr\"odinger equation
with the initial data $\psi(x,0)=\psi(x)$ and vice versa.
This extension of Schr\"odinger's form establishes a mathematical expression of wave-particle duality and that von Neumann's entropy
is a quantitative measure of complementarity between wave-like and particle-like behaviors.
Furthermore, it suggests that collapses of Schr\"odinger's wave functions are just the simultaneous transition of the
particle from many levels to one. See \cite{C} for details.
The equation considered as a wave function equation is also explicitly used in \cite{KM}, connecting with Bose-Einstein condensation.

The problem we want to discuss in this paper is integrability of wave propagator $e^{it(\Delta_x -\Delta_y)}$
which gives a formula for the solution to the Liouville-von Neumann equation.
Applying the Fourier transform to \eqref{nle0}, the solution $\Psi(x,y,t)$ is indeed given by
$$e^{it(\Delta_x -\Delta_y)}f(x, y):= \frac{1}{(2\pi)^{2n}} \int_{\mathbb{R}^n} \int_{\mathbb{R}^n} e^{ix\cdot\xi + iy\cdot\xi' -it( |\xi|^2-|\xi'|^2)} \hat{f}(\xi,\xi')  d\xi d\xi',$$
where $f$ is the initial data $\Psi(x,y,0)$ and  $\hat{f}$ is the Fourier transform thereof.
Our result is stated as follows.

\begin{thm}\label{thm}
Let $n\ge 1$. Assume that $2 \leq q \leq \infty$, $2 \leq r_2 \leq r_1 \leq \infty$ and
\begin{equation}\label{admi}
\frac{2}{q}=n\Big(1-\frac{1}{r_1}-\frac{1}{r_2} \Big).
\end{equation}
Then we have
\begin{equation}\label{T}
\|e^{it(\Delta_x -\Delta_y)}f\|_{L_t^q  L_x^{r_1} L_y^{r_2}} \lesssim \|f\|_{L^2},
\end{equation}
except for $(q, r_1, r_2)=(2, \infty, \infty)$ when $n=1$ and for $(q, r_1, r_2)=(2, \infty, 2)$ when $n=2$.
\end{thm}

Compared with \eqref{T}, the space-time integrability known as \textit{Strichartz estimates} for the Schr\"odinger case has been extensively studied over the last several decades and is now completely understood as follows (see \cite{St,GV,M-S,KT}):
\begin{equation*}
\|e^{it\Delta} f\|_{L_t^q L_x^r} \lesssim \|f\|_{L^2}
\end{equation*}
if and only if $(q,r)$ is $n$-Schr\"odinger-admissible, i.e.,
$q,r\geq 2$, $(q,r,n)\neq(2,\infty,2)$ and $2/q + n/r = n/2$.
Particularly when $r_1=r_2$, the exponent pair $(q,r_1,r_2)$ in the theorem becomes $2n$-Schr\"odinger-admissible.
In this case, \eqref{T} can be found in \cite{LL}
and the range of $(r_1,r_2)$ is given by the closed segment $[D,B]$ in Figure \ref{fig1} below.
But, it is significant in the \textit{bipartite} form to quantify the spatial integrability differently with respect to $x$ and $y$.
In this regard, the main contribution of the theorem is to extend the diagonal case $r_1=r_2$ to mixed norms $L_x^{r_1} L_y^{r_2}$, $r_1\neq r_2$,
where the range of $(r_1,r_2)$ is given by the closed triangle with vertices $A,D,B$.

Notice that the condition \eqref{admi} is necessary for \eqref{T} to be invariant
under the scaling $(x, y,t) \rightarrow (\lambda x, \lambda y,\lambda^2t)$, $\lambda >0$.
By the standard $TT^*$ method, \eqref{T} is also equivalent to the boundedness of the time translation invariant operator
$TT^\ast: F \rightarrow \int_{\mathbb{R}} e^{i(t-s)(\Delta_x-\Delta_y)}F(s) ds $
from $L_t^q L_x^{r_1} L_y^{r_2}$ to $L_t^{q'} L_x^{r_1'} L_y^{r_2'}$.
Hence, $q \ge q'$ (i.e., $q \ge 2$) is required (see \cite{H} or \cite{G}).
When $q=2$ and $q=\infty$, $(1/r_1,1/r_2)$ lies on the line through the points $A,D$
and on the point $B$, respectively.

\begin{figure}
	\centering
	\includegraphics[width=0.5\textwidth]{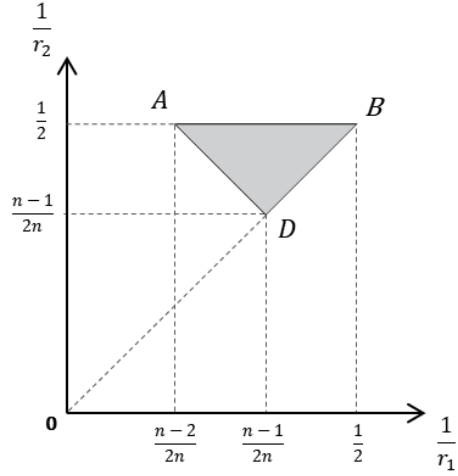}
	\label{fig1}
   \caption{The range of $(r_1,r_2)$ for which \eqref{T} holds when $n\geq3$.}
\end{figure}

In the final section we also give some applications of Theorem \ref{thm} to nonlinear problems.

\

Throughout this paper, we use $\mathcal{F}f$ and $\hat{f}$ to denote the Fourier transform of $f$
and $\langle f, g \rangle_{x,y}$ denotes the usual inner product on $L_{x,y}^2$.
We also denote A $\lesssim$ B to mean $A \leq CB$ with unspecified constant $C>0$ which may be different at each occurrence.

\section{The non-endpoint case $q>2$}\label{sec3}
In this section we prove Theorem \ref{thm} when $q>2$.
The endpoint case $q=2$ will be proved in the next section.

When $q=\infty$, the estimate \eqref{T} follows directly from the Plancherel theorem as follows:
\begin{align*}
\| e^{it(\Delta_x-\Delta_y)}f \|_{L_t^{\infty} L_{x,y}^{2}}
= \sup_{t \in \mathbb{R}} \| e^{-it(|\xi|^2-|\xi'|^2)}  \hat{f} \|_{L_{\xi, \xi'}^2} = \| \hat{f} \|_{L_{\xi, \xi'}^2} =\| f \|_{L_{x,y}^2}.
\end{align*}
Now we only need to consider $2 < q < \infty$.
By the standard $TT^{\ast}$ argument, we may prove the following estimate
\begin{equation}\label{TT*}
\left\| \int_{\mathbb{R}} e^{i(t-s)(\Delta_x-\Delta_y)} F(s) ds \right\|_{L_t^q L_x^{r_1} L_y^{r_2}} \lesssim \left\| F \right\|_{L_t^{q'} L_x^{r_1'} L_y^{r_2'}}
\end{equation}
which is equivalent to \eqref{T}.
To show this, we obtain the following fixed-time estimates for the propagator $e^{it(\Delta_x -\Delta_y)}$.

\begin{lem}\label{lem}
Let $n \ge 1$ and $2 \leq r_2 \leq r_1 \leq \infty$. Then we have
\begin{equation}\label{decay}
\|e^{it(\Delta_x -\Delta_y)}G\|_{L_x^{r_1} L_y^{r_2}} \lesssim |t|^{-n\big(1-\frac{1}{r_1}-\frac{1}{r_2}\big)} \| G\|_{L_x^{r_1'} L_y^{r_2'}}.
\end{equation}
\end{lem}

Assuming for the moment this lemma, we see that
\begin{align*}
\bigg\| \int_{\mathbb{R}}e^{i(t-s)(\Delta_x -\Delta_y)} F(s) ds \bigg\|_{L_t^q L_x^{r_1} L_y^{r_2}}
& \leq\bigg\| \int_{\mathbb{R}}\| e^{i(t-s)(\Delta_x -\Delta_y)} F(s)\|_{L_x^{r_1} L_y^{r_2}} ds \bigg\|_{L_t^q} \cr
&\lesssim \bigg\|  \int_{\mathbb{R}} |t-s|^{-n\big(1-\frac{1}{r_1}-\frac{1}{r_2}\big)} \| F(s)\|_{L_x^{r_1'} L_y^{r_2'}} ds\bigg\|_{L_t^q}.
\end{align*}
Here we use the Hardy-Littlewood-Sobolev inequality (\cite{S}, Section V.1.2) for dimension one,
\begin{align}\label{hls}
\big\| |t|^{-\alpha} \ast g\big\|_{L^q} \lesssim \|g\|_{L^{p}},
\end{align}
where $0 < \alpha <1$, $1 \leq p < q < \infty $ and $ 1/q+1=1/p+\alpha$.
By applying \eqref{hls} with $p=q'$ and $\alpha=n(1-1/r_1-1/r_2)$ to the above,
we obtain the estimate \eqref{TT*}
if $2 <q <\infty$, $2\leq r_1 \leq r_2 \leq \infty$ and $2/q=n(1-1/r_1-1/r_2)$, as desired.

\begin{proof}[Proof of Lemma \ref{lem}]
It remains to prove the lemma. By the Riesz-Thorin interpolation theorem, it suffices to show \eqref{decay} for the following three cases:
\begin{itemize}
\item[(a)] $r_1=r_2=\infty$,
\item[(b)] $r_1=\infty$ and $r_2=2$,
\item[(c)] $r_1=r_2=2$.
\end{itemize}
First we write
\begin{align*}
e^{it(\Delta_x -\Delta_y)} G(x,y)
&= \int_{\mathbb{R}^n} \int_{\mathbb{R}^n} K_{t} (x-x') K_{-t} (y-y') G(x', y') dx' dy' \cr
&= K_{-t} \ast_y ( K_{t} \ast_x G )
\end{align*}
where
$$K_{t}(x)=\frac{1}{(4\pi it)^{n/2}} e^{\frac{i|x|^2}{4t}}$$
denotes the integral kernel for the Schr\"odinger propagator.
Then one can see that
\begin{equation}\label{kernel_esti}
\|K_{t}\ast g \|_{L^\infty}
\lesssim |t|^{-\frac{n}{2}}\|g\|_{L^1}
\end{equation}
by Young's inequality and that
\begin{equation}\label{kernel_esti2}
\|K_{t}\ast g \|_{L^2}=\|e^{it\Delta} g \|_{L^2}=\|g\|_{L^2}
\end{equation}
by the Plancherel theorem.

Applying \eqref{kernel_esti} repeatedly together with the Minkowski inequality, we obtain the first case $(a)$ as follows:
\begin{align*}
\|e^{it(\Delta_x -\Delta_y)} G \|_{L_{x,y} ^{\infty}}
&=\big\| \| K_{-t} \ast_y ( K_{t} \ast_x G ) \|_{L_y ^\infty} \big\|_{L_x^\infty}\cr
&\lesssim |t|^{-\frac{n}{2}} \big\| \|  K_{t} \ast_x G  \|_{L_y^1} \big\|_{L_x ^\infty}\cr
&\lesssim |t|^{-\frac{n}{2}} \big\| \|  K_{t} \ast_x G  \|_{L_x^\infty} \big\|_{L_y ^1}\cr
&\lesssim |t|^{-n} \|  G  \|_{L_{x,y}^1}.
\end{align*}
For the second case $(b)$, we use \eqref{kernel_esti2} and \eqref{kernel_esti}
along with the Minkowski inequality to get
\begin{align*}
\|e^{it(\Delta_x -\Delta_y)} G \|_{L_x^{\infty} L_y^{2}}
&= \big\| \|  K_{-t} {\ast}_y  (K_{t} \ast_x G )  \|_{L_y^2} \big\|_{L_x ^\infty}\cr
&\lesssim \big\| \|   K_{t} {\ast}_x G  \|_{L_y^2} \big\|_{L_x ^\infty}\cr
&\lesssim \big\| \| K_{t} {\ast}_x G \|_{L_x^\infty} \big\|_{L_y ^2}\cr
&\lesssim |t|^{-\frac{n}{2}} \big\| \| G\|_{L_x^1} \big\|_{L_y^2}\cr
&\lesssim |t|^{-\frac{n}{2}} \|  G  \|_{L_x^1 L_y^2}.
\end{align*}
Lastly the case $(c)$ follows directly from the Plancherel theorem as
\begin{align*}
\|e^{it(\Delta_x -\Delta_y)} G \|_{L_{x,y} ^{2}}
= \big\| e^{-it(|\xi|^2-|\xi'|^2) }   \widehat{G} \big\|_{L_{\xi,\xi'} ^2}
=\| G \| _{L_{x,y}^2}.
\end{align*}
\end{proof}

\section{The endpoint case $q=2$}
It remains to prove \eqref{T} when $q=2$.
Following \cite{KT}, we will obtain the estimate by a bilinear interpolation between the nonendpoint results and the decay estimates
for a time-localized bilinear form operator.
In this argument we can take advantage of the symmetry and the flexibility of the bilinear form setting.

By the standard $TT^*$ method we may prove
$$\bigg\|\int_{\mathbb{R}} e^{i(t-s)(\Delta_x-\Delta_y)}F(s)ds\bigg\|_{L_t^qL_x^{r_1}L_y^{r_2}}\lesssim\|F\|_{L_t^{q'}L_x^{r_1'}L_y^{r_2'}},$$
and by duality this is in turn equivalent to the bilinear form estimate
\begin{equation*}
\int_{\mathbb{R}} \int_{\mathbb{R}} \left\langle e^{-is(\Delta_x -\Delta_y)} F(s),\, e^{-it(\Delta_x -\Delta_y)} G(t) \right\rangle_{x,y} ds dt \lesssim \|F\|_{L_t^{2} L_x^{r_1'} L_y^{r_2'}} \|G\|_{L_t^{2} L_x^{r_1'} L_y^{r_2'}}.
\end{equation*}
By symmetry it suffices to restrict our attention to the retarded region
$$\Omega=\{ (s, t)\in \mathbb{R}^2 :s<t \}$$
in the above double integral.
Now we break $\Omega$ into a series of time-localized regions using a Whitney type decomposition (see \cite{S} or \cite{F});
let $\mathcal{Q}_j$ be the family of dyadic squares in $\Omega$ whose side length is dyadic number $2^j$ for $j\in \mathbb{Z}$.
Each square $Q =I\times J \in \mathcal{Q}_j$ has the property that
\begin{equation}\label{whit}
 2^j\sim |I| \sim |J| \sim \textnormal{dist}(I, J)
\end{equation}
and $\Omega = \cup_{j\in \mathbb{Z}} \cup_{Q \in \mathcal{Q}_j}Q$ where the squares $Q$ are essentially disjoint.

Hence we are reduced to showing the following estimate
\begin{align}\label{local2}
\sum_{j\in \mathbb{Z}}\sum_{Q \in \mathcal{Q}_j}  |T_{j}(F, G)| \lesssim \|F\|_{L_t^{2} L_x^{r_1'} L_y^{r_2'}} \|G\|_{L_t^{2} L_x^{r_1'} L_y^{r_2'}}
\end{align}
where
\begin{equation}\label{tj}
T_{j}(F, G) :=   \int_J \int_I  \left\langle e^{-is(\Delta_x -\Delta_y)} F(s),\ e^{-it(\Delta_x -\Delta_y)} G(t) \right\rangle_{x,y} ds dt.
\end{equation}

To get this estimate, we make use of the following two-parameter family of estimates in which the case $r_1=r_2$ is excluded
but this is harmless because the estimate \eqref{T} is already known for this case (\cite{LL}).

\begin{prop}\label{lem2}
Let $n\ge1$. Assume that $2 \leq r_2 < r_1 \leq \infty$ and
\begin{equation*}
 1=n(1-\frac{1}{r_1}-\frac{1}{r_2}).
\end{equation*}
Then we have
\begin{equation}\label{lemeq2}
 \sum_{Q \in \mathcal{Q}_j} \left| T_{j}(F, G) \right| \lesssim 2^{-j\beta( a, \tilde{a})} \|F \|_{L_t^{2} L_x^{a'} L_y^{r_2'}} \|G\|_{L_t^{2}L_x^{\tilde{a}'} L_y^{r_2'}}
\end{equation}
for all $j \in \mathbb{Z}$ and all $(1/a, 1/\tilde{a})$ in a neighborhood of $(1/r_1, 1/r_1)$ (see Figure 2)
with $$\beta(a, \tilde{a})=-1+\frac{n}{2}(2-\frac{1}{a}-\frac{1}{\tilde{a}}-\frac{2}{r_2}).$$
\end{prop}

Assuming for the moment this proposition which will be obtained in the next section,
we now get \eqref{local2} using a bilinear interpolation argument.
Let us first consider the bilinear vector-valued operator $B$ as
$$B(F, G)=\bigg\{  \sum_{Q \in \mathcal{Q}_j} T_{j}(F, G) \bigg\}_{j \in \mathbb{Z}}.$$
Then, \eqref{local2} can be rewritten as
\begin{equation}\label{end2}
B: L_t^{2} L_x^{r_1'} L_y^{r_2'} \times L_t^{2} L_x^{r_1'} L_y^{r_2'} \rightarrow \ell_1^0
\end{equation}
where  $\ell_p^{\alpha}$ for $\alpha \in \mathbb{R}$ and $1 \leq p \leq \infty$ denotes the weighted sequence space with the norm
\begin{align*}
\|\{ x_j \}_{j \ge 0} \|_{\ell_p^{\alpha}} =
\begin{cases}
  (\sum_{j \ge 0} 2^{j\alpha p} |x_j|^p )^{\frac{1}{p}}, & \mbox{if }\, p \neq \infty, \\
  \sup_{j \ge 0} 2^{j\alpha} |x_j|, & \mbox{if }\, p =\infty.
\end{cases}
\end{align*}
For a sufficient small $\varepsilon>0$, we now choose $\frac{1}{a_0}=\frac{1}{r_1} -\varepsilon$ and
$\frac{1}{a_1}=\frac{1}{r_1}+2\varepsilon$.
Note here that we cannot choose $a_0,a_1\geq0$ if $r_1=\infty$ which corresponds to the cases where $(q,r_1,r_2)=(2,\infty,\infty)$ when $n=1$
and $(q,r_1,r_2)=(2,\infty,2)$ when $n=2$.
For this reason these cases are excluded in the theorem.
Since $\beta(a_0, a_0) =2n\varepsilon$ and $\beta( a_0, a_1) =\beta(a_1, a_0) =-n\varepsilon$,
Proposition \ref{lem2} implies
\begin{align*}
	& B :\ L_t^{2} L_x^{a_0'} L_y^{r_2'}\times L_t^{2} L_x^{a_0'} L_y^{r_2'}\rightarrow \ell_{\infty} ^{2n\varepsilon}, \cr
	& B :\ L_t^{2} L_x^{a_0'} L_y^{r_2'}\times L_t^{2} L_x^{a_1'} L_y^{r_2'}\rightarrow \ell_{\infty} ^{-n\varepsilon}, \cr
	& B :\ L_t^{2} L_x^{a_1'} L_y^{r_2'}\times L_t^{2} L_x^{a_1'} L_y^{r_2'}\rightarrow \ell_{\infty} ^{-n\varepsilon}.
\end{align*}

Now we apply the following bilinear interpolation lemma with $s =1$, $p=q=r_1'$ and $\theta_0=\theta_1=1/3$ to obtain
\begin{equation}\label{bibi}
	B :  (L_t^{2} L_x^{a_0'} L_{y}^{r_2'}, L_t^{2} L_x^{a_1'} L_{y}^{r_2'})_{\frac{1}{3},\,r_1'} \times (L_t^{2} L_x^{a_0'} L_{y}^{r_2'}, L_t^{2} L_x^{a_1'} L_{y}^{r_2'})_{\frac{1}{3},r_1'}
\rightarrow(\ell_{\infty} ^{2n\varepsilon}, \ell_{\infty} ^{-n\varepsilon})_{\frac{2}{3},1}.
\end{equation}
Here, $(\cdot\,,\cdot)_{(\theta,p)}$ denotes the real interpolation functor.
\begin{lem}\label{biinter} (\cite{BL}, Section 3.13.5(b))
Let $A_0,A_1,B_0, B_1, C_0, C_1$ are Banach spaces and $B$ be a bilinear operator such that $B: A_0\times B_0\rightarrow C_0$,
$B: A_0\times B_1\rightarrow C_1$ and $B: A_1\times B_0\rightarrow C_1$. Then
$$B: (A_0, A_1)_{\theta_0,ps} \times(B_0, B_1)_{\theta_1,qs}\rightarrow(C_0, C_1)_{\theta,s}$$
if\, $0<\theta_0, \theta_1 < \theta=\theta_0 +\theta_1$, $1 \leq p,q,s \leq \infty$ and $1 \leq 1/p + 1/q$.
\end{lem}

Finally by making use of the real interpolation space identities in Lemma \ref{ide} below,
we easily see that
\begin{align*}
(L_t^{2} L_x^{a_0'} L_{y}^{r_2'}, L_t^{2} L_x^{a_1'} L_y^{r_2'})_{1/3,\,r_1'} = L_t^{2} L_x^{r_1'} L_y^{r_2'}
\end{align*}
and
$$(\ell_{\infty} ^{2n\varepsilon}, \ell_{\infty} ^{-n\varepsilon})_{2/3,\,1} = \ell_{1} ^0$$
in \eqref{bibi}.
This implies \eqref{end2} as desired.

\begin{lem}\label{ide}(\cite{BL})
Let $0< \theta < 1$, $1\leq p_0, p_1 \leq\infty$ and $s_0, s_1 \in \mathbb{R}$.
Then
$$(L^{p_0}(A_0), L^{p_1}(A_1))_{\theta, p} = L^{p}((A_0, A_1)_{\theta, p})$$
for two complex Banach spaces $A_0,A_1$ and $1/p=(1-\theta)/p_0+\theta/p_1$, and
	\begin{equation*}
	(\ell_{p_{0}}^{s_{0}},\, \ell_{p_{1}}^{s_{1}})_{\theta, p} = \ell_{p}^{s}
	\end{equation*}
if $s_{0} \neq s_{1}$ and $s=(1-\theta)s_{0} +\theta s_{1}$.
\end{lem}

\section{Proof of Proposition \ref{lem2}}

To get \eqref{lemeq2}, we only need to show
\begin{align}\label{ex2}
|T_{j}(F, G)| \lesssim 2^{-j\beta( a,\tilde{a})} \|F\|_{L_t^{2}(I; L_x^{a'} L_y^{r_2'})} \|G\|_{L_t^{2}(J; L_x^{\tilde{a}'} L_y^{r_2'})}
\end{align}
for each square $Q =I\times J \in \mathcal{Q}_j$.
Using the fact that for each $I$ there are at most a fixed finite number of intervals $J$ which satisfy \eqref{whit} and they are all contained
in a neighborhood of $I$ of size $O(2^j)$, we indeed get
\begin{align*}
 \sum_{Q \in \mathcal{Q}_j} \left|T_{j}(F, G) \right| & \lesssim 2^{-j\beta(a,\tilde{a})}\sum_{Q \in \mathcal{Q}_j}  \|F \|_{L_t^{2}(I; L_x^{a'} L_y^{r_2'})} \|G\|_{L_t^{2} (I;L_x^{\tilde{a}'} L_y^{r_2'})} \cr
& \leq  2^{-j\beta(a, \tilde{a})} \bigg(\sum_{Q \in \mathcal{Q}_j } \|F \|_{L_t^{2}(I; L_x^{a'} L_y^{r_2'})}^{2}  \bigg)^{\frac{1}{2}} \cdot \bigg(\sum_{Q \in \mathcal{Q}_j} \|G\|_{L_t^{2} (I;L_x^{\tilde{a}'} L_y^{r_2'})}^{2} \bigg)^{\frac{1}{2}}\cr
&\lesssim  2^{-j\beta(a, \tilde{a})}\|F \|_{L_t^{2}(\mathbb{R}; L_x^{a'} L_y^{r_2'})}\|G \|_{L_t^{2}(\mathbb{R}; L_x^{a'} L_y^{r_2'})}
\end{align*}
as desired.

We shall now show \eqref{ex2} for the following three cases (see Figure 2):
\begin{itemize}
\item[$(a) $] $a=\tilde{a} =\infty$ (point $E$),
\item[$(b) $] $r_2\leq a<r_1$ and $\tilde{a}=2$ (segment $(A,B]$),
\item[$(c) $] $a=2$ and $r_2\leq \tilde{a}<r_1$ (segment $(D,C]$).
\end{itemize}
The proposition will then follow by interpolation with the range of $(a, \tilde{a})$ as in Figure 2.

\begin{figure}\label{fig2}
\centering
\includegraphics[width=0.45\textwidth]{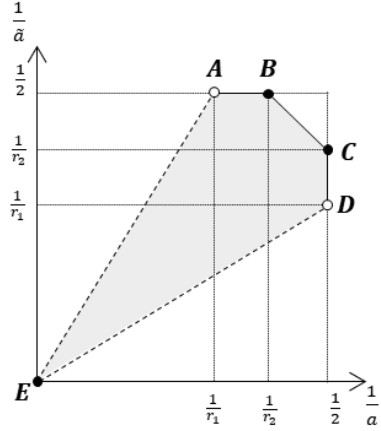}
\caption{The range of $(a, \tilde{a})$ for which \eqref{lemeq2} holds when $n\ge3$.}
\end{figure}

To show $(a)$,
we first use H\"older's inequality in $x, y$ and the decay estimate \eqref{decay}
with the fact that $|t-s|\sim 2^j$;
\begin{align*}
\left|T_{j}(F, G) \right|& =
\left| \int_J \int_I \left\langle e^{i(t-s)(\Delta_x -\Delta_y)} F(s),\  G(t) \right\rangle_{x,y} ds dt \right| \cr
&\leq  \int_J \int_I  \| e^{i(t-s)(\Delta_x -\Delta_y)} F(s) \|_{L_x^{\infty}L_y^{r_2}}  \| G(t)\|_{L_x^{1} L_y^{r_2'}}  ds dt \cr
&\lesssim 2^{-nj(1-\frac{1}{r_2})} \|F\|_{L_t^1(I;L_x^{1} L_y^{r_2'})} \| G\|_{L_t^1(J;L_x^{1} L_y^{r_2'})},
\end{align*}
and then H\"older's inequality in each $t, s$ gives
$$ \left|T_{j}(F, G) \right| \leq 2^{-nj(1-\frac{1}{r_2})+j}  \|F\|_{L_t^2(I;L_x^{1} L_y^{r_2'})} \| G\|_{L_t^2(J;L_x^{1} L_y^{r_2'})}$$
as desired.

For the case $(b)$, we bring the $s$-integration inside the inner product in \eqref{tj} and apply H\"older's inequality in $x, y$ to obtain
\begin{align}\label{eq2}
\left|T_{j}(F, G) \right|
& \leq \int_J  \left\| \int_I e^{-is(\Delta_x -\Delta_y)} F(s) ds \right\|_{L_{x,y}^2}  \big\|e^{-it(\Delta_x -\Delta_y)} G(t)\big\|_{L_{x,y}^2} dt  \cr
& \leq \left\| \int_I e^{-is(\Delta_x -\Delta_y)} F(s) ds \right\|_{L_{x,y}^2}  \int_J  \big\|e^{-it(\Delta_x -\Delta_y)} G(t) \big\|_{L_{x,y}^{2}} dt.
\end{align}
By applying the Plancherel theorem and H\"older's inequality, the second term in the right-hand side of \eqref{eq2} becomes
\begin{equation}\label{gjk}
 \int_{J} \big\|e^{-it(\Delta_x -\Delta_y)} G(t)\big\|_{L_{x,y}^{2} } dt \leq 2^{\frac{j}{2}} \| G\|_{L_t^2(J;L_{x,y}^2)}.
 \end{equation}
On the other hand, we handle the first term using the dual version of the nonendpoint estimates as follows:
\begin{align*}
\left\| \int_I e^{-is(\Delta_x -\Delta_y)} F(s)\ ds \right\|_{L_{x,y}^2}
&=\left\| \int_{\mathbb{R}} e^{-is(\Delta_x -\Delta_y)} \chi_I(s)F(s)\ ds \right\|_{L_{x,y}^2}\cr
&\lesssim \left\|\chi_{I} F  \right\|_{L_t^{q_a'} L_x^{a'} L_y^{r_2'}}\cr
&=\left\| F  \right\|_{L_t^{q_a'}(I; L_x^{a'} L_y^{r_2'})},
\end{align*}
where $2 < q_a \leq \infty$, $2\leq r_2 \leq a$ and $2/q_a=n(1-1/a-1/r_2)$ which imply
\begin{equation*}
\frac{1}{r_1} =1-\frac1{r_2}-\frac1n< \frac{1}{a}\leq \frac{1}{r_2}.
\end{equation*}
Then H\"older's inequality in $t$ gives
\begin{align}\label{dfg}
\nonumber\left\| \int_I e^{-is(\Delta_x -\Delta_y)} F(s)\ ds \right\|_{L_{x,y}^2}
\nonumber&\lesssim 2^{j(\frac{1}{2}-\frac{1}{q_a})} \left\| F  \right\|_{L_t^{2}(I; L_x^{a'} L_y^{r_2'})} \\
&= 2^{j(\frac{1}{2}-\frac{n}{2}(1-\frac{1}{a}-\frac{1}{b}))} \left\| F  \right\|_{L_t^{2}(I; L_x^{a'} L_y^{r_2'})}.
\end{align}
Combining \eqref{eq2}, \eqref{gjk} and \eqref{dfg}, we finally get
$$\left|T_{j}(F, G) \right| \lesssim 2^{j(1 -\frac{n}{2}(1-\frac{1}{a}-\frac{1}{r_2}))}
\| F  \|_{L_t^{2}(I; L_x^{a'} L_y^{r_2'})} \| G\|_{L_t^{2}(J;L_{x,y}^2)} $$
as desired.
A similar argument gives the case $(c)$.

\section{Applications}
In this final section we present some applications of Theorem \ref{thm} to the following nonlinear problem:
\begin{equation}\label{lv}
\begin{cases}
i\partial_tu+(\Delta_x-\Delta_y)u=\pm|u|^{\alpha}u, \\
u(x, y, 0)=f(x, y),
\end{cases}
\end{equation}
where $(x,y,t)\in\mathbb{R}^n \times \mathbb{R}^n \times \mathbb{R}$ and $\alpha>0$.
Particularly when $n=1$, this equation with $\alpha=2$ is reduced to the two-dimensional hyperbolic nonlinear Schr\"odinger equation
which appears in nonlinear optics (\cite{DLS,SS}) and aries naturally in the study of modulation of wave trains in gravity
water waves (\cite{T,TW}).
Various issues concerning the well-posedness of this special case $n=1$ have been intensively studied until lately
(see, for example, \cite{CG,DMPS} and references therein).
Motivated by this, we address here the higher dimensional cases.
Our result is the following theorem which shows that the problem \eqref{lv} is locally well-posed:

\begin{thm}\label{thm2}
	Let $n \ge3$.
	Then for $f \in L_x^2 H_y^1$ there exist $T>0$ and a unique solution to \eqref{lv} for $0<\alpha<2/(n-1)$,
	$$ u \in C_t([0, T]; L^2_x H^1_y) \cap L_t^q ([0,T];L_x^{r_1} W_y^{1, r_2}),$$
with $(q, r_1, r_2)$ satisfying all the conditions given in Theorem \ref{thm} together with
\begin{equation}\label{res}
\frac{\alpha(n-1)-1}{2}\leq \frac{1}{q}\leq \frac{\alpha(n-1)}{2},\quad
\frac{1}{2}-\frac{\alpha}{2}\leq \frac{1}{r_2}-\frac{\alpha}{n}\leq \frac{1}{r_1} \leq 1-\frac{\alpha}{2}.
\end{equation}
\end{thm}

\begin{rem}
From the proof, one can see that the same result can be valid for $L^2$ initial data $f \in L_x^2 L_y^2$
under the diagonal case $r_1=r_2$ with $0<\alpha<2/n$.
\end{rem}

\begin{proof}[Proof of Theorem \ref{thm2}]
By Duhamel's principle, the solution of \eqref{lv} can be written as
\begin{equation}\label{duh}
\Phi(u)=e^{it (\Delta_x -\Delta_y)} f -i \int_0^t e^{i(t-s)(\Delta_x - \Delta_y)} F(u) ds
\end{equation}
where $F(u)= \pm|u( \cdot, \cdot, s)|^{\alpha} u( \cdot, \cdot, s).$
For suitable values of $T, A>0$, we shall show that $\Phi$ defines a contraction map on
\begin{align*}
X(T, A)=  \bigg\lbrace u \in C_t(I; L^2_x H^1_y)& \cap L_t^q (I ;L_x^{r_1} W_y^{1, r_2}): \\
&\sup_{t\in I} \|u\|_{L_x^2 H_y^1} +\|u\|_{L_t^q(I;L_x^{r_1} W_y^{1, r_2})} \leq A \bigg\rbrace
\end{align*}
equipped with the distance
$$d(u, v)= \sup_{t\in I} \|u-v\|_{L_{x,y}^2} +\|u-v\|_{L_t^q(I; L_x^{r_1} L_y^{r_2})},$$
where $I=[0, T]$ and $(q, r_1, r_2)$ is given as in Theorem \ref{thm2}.

To control the Duhamel term in \eqref{duh}, we need the following inhomogeneous estimates
which are derived from the homogeneous estimates \eqref{T} adopting $TT^{\ast}$ argument and the Christ-Kiselev lemma \cite{CK}:

\begin{cor}\label{cor}
Let $n\ge 1$. Assume that $(q,r_1,r_2)$ and $(\tilde{q},\tilde{r}_1,\tilde{r}_2)$ are given as in Theorem \ref{thm}.
Then
\begin{equation}\label{TT}
\left\| \int_0^t e^{it(\Delta_x -\Delta_y)} F(s)\, ds \right\|_{L_t^q  L_x^{r_1} L_y^{r_2}} \lesssim \|F\|_{L_t^{\tilde{q}'} L_x^{\tilde{r}_1'} L_y^{\tilde{r}_2'}}
\end{equation}
if $q>\tilde{q}'$.
\end{cor}

Now we first show that $\Phi(u) \in X$ for $u\in X$.
Using Plancherel's theorem and then the adjoint form of \eqref{T}, we see that
\begin{align*}
\sup_{t\in I}\|\Phi(u)\|_{L_x^2 H_y^1}
&\leq C\|f\|_{L_x^2 H_y^1}+C\sup_{t\in I}\bigg\| \int_0^t e^{-is(\Delta_x - \Delta_y)} F(u) ds\bigg\|_{L_x^2 H_y^1}\\
&\leq C\|f\|_{L_x^2 H_y^1}+C\| F(u)\|_{L_t^{\tilde{q}'}(I;L_x^{\tilde{r}_1'}W_y^{1, \tilde{r}_2'})}
\end{align*}
for $(\tilde{q}, \tilde{r}_1, \tilde{r}_2)$ given as in Corollary \ref{cor}.
For $q>\tilde{q}'$, we assume for the moment that
\begin{equation}\label{no}
\|F(u)\|_{L_t^{\tilde{q}'}(I;L_x^{\tilde{r}_1'} W_y^{1, \tilde{r}_2'})}
\leq CT^{\frac{1}{\tilde{q}'}-\frac{1}{q}} \|u\|_{L_t^{\infty}(I;L_x^2 H_y^1) }^{\alpha} \|u\|_{L_t^{q}(I;L_x^{r_1}W_y^{1, r_2})},
\end{equation}
where
\begin{equation}\label{hol}
1/\tilde{r}_1'=\alpha/2 + 1/r_1\quad\text{and}\quad 1/\tilde{r}_2'=\alpha(1/2-1/n)+1/r_2
\end{equation}
which imply
\begin{equation}\label{alp}
1/\tilde{q}'-1/q=(2+\alpha-n\alpha)/2>0\quad (\text{thus}, \alpha<2/(n-1))
\end{equation}
by combining the conditions \eqref{admi} for $(q,r_1,r_2)$ and $(\tilde{q},\tilde{r}_1,\tilde{r}_2)$.
Hence we get for $u\in X$
\begin{equation*}
\sup_{t\in I}\|\Phi(u)\|_{L_x^2 H_y^1}\leq C\|f\|_{L_x^2 H_y^1}+ CT^{\frac{2+\alpha-n\alpha}{2}} A^{\alpha+1}.
\end{equation*}
The same argument together with Theorem \ref{thm} also implies
$$\|\Phi(u)\|_{L_t^{q}(I; L_x^{r_1} W_y^{1, r_2})}\leq C\|f\|_{L_x^2 H_y^1} +CT^{\frac{2+\alpha-n\alpha}{2}} A^{\alpha+1}$$
for $u\in X$.
Consequently, $\Phi(u)\in X$ if
\begin{equation}\label{ne}
C\|f\|_{L_x^2 H_y^1}+CT^{\frac{2+\alpha-n\alpha}{2}}A^{\alpha+1} \leq A.
\end{equation}

Next we show that for $u, v \in X$
\begin{equation}\label{cont}
d\left( \Phi(u), \Phi(v) \right) \leq \frac{1}{2} d(u, v).
\end{equation}
First we use \eqref{TT} to get
\begin{align*}
d\left( \Phi(u), \Phi(v) \right)
&= d\left(\int_0^t e^{i(t-s)(\Delta_x -\Delta_y)} F(u)\, ds ,\int_0^t e^{i(t-s)(\Delta_x -\Delta_y)} F(v)\, ds  \right)\\
& \leq  C\| F(u)-F(v)\|_{L_t^{\tilde{q}'}(I;L_x^{ \tilde{r}_1'} L_y^{\tilde{r}_2'})}.
\end{align*}
Using the simple inequality
$||u|^{\alpha}u-|v|^{\alpha}v|\leq C(|u|^{\alpha}+|v|^{\alpha})|u-v|$,
and then applying H\"older's inequality in $x,y,t$ under the condition \eqref{hol}, we obtain
\begin{align*}
\|F(u)-&F(v)\|_{L_t^{\tilde{q}'}(I;L_x^{ \tilde{r}_1'} L_y^{\tilde{r}_2'})} \\
&\leq CT^{\frac{1}{\tilde{q}'}-\frac{1}{q}}
(\|u\|_{L_t^{\infty}(I; L_x^2 L_y^{\frac{2n}{n-2}})}^{\alpha}+\|v\|_{L_t^{\infty}(I; L_x^2 L_y^{\frac{2n}{n-2}})}^{\alpha})\|u-v\|_{L_t^q(I; L_x^{r_1} L_y^{r_2})}.
\end{align*}
Then, \eqref{alp} and the Sobolev embedding $\dot{H}^{1} \hookrightarrow L^{\frac{2n}{n-2}}$ when $n>2$ give
$$ d\left( \Phi(u), \Phi(v) \right) \leq CT^{\frac{2+\alpha-n\alpha}{2}}A^{\alpha}\, d(u, v)$$ for $u\in X$.

Finally, we choose $A=2C\|f\|_{L_x^2 H_y^1}$ and $T$ so that $CT^{\frac{2+\alpha-n\alpha}{2}}A^{\alpha} \leq 1/2$,
and thus \eqref{ne} and \eqref{cont} hold as desired.
Therefore, there exists a unique local solution $u\in C_t(I; L^2_x H^1_y) \cap L_t^q (I ;L_x^{r_1} W_y^{1, r_2})$.

It remains to show \eqref{no}. We first see that
\begin{equation*}
\|F(u)\|_{L_t^{\tilde{q}'}(I;L_x^{\tilde{r}_1'} W_y^{1, \tilde{r}_2'})}
\leq C\| |u|^{\alpha}u\|_{L_t^{\tilde{q}'}(I;L_x^{\tilde{r}_1'} L_y^{ \tilde{r}_2'})}
 +C\| |u|^{\alpha} |\nabla_y u|\|_{L_t^{\tilde{q}'}(I;L_x^{\tilde{r}_1'} L_y^{ \tilde{r}_2'})}.
\end{equation*}
By H\"older's inequality, we then get
$$\| |u|^{\alpha}u\|_{L_t^{\tilde{q}'}(I;L_x^{\tilde{r}_1'} L_y^{ \tilde{r}_2'})} \leq C\|u \|_{L_t^{\infty}(I;L_x^2 L_y^{\frac{2n}{n-2}})}^{\alpha} \|u \|_{L_t^{\tilde{q}'}(I;L_x^{r_1}L_y^{r_2})}
$$
and
$$\| |u|^{\alpha} |\nabla_y u|\|_{L_t^{\tilde{q}'}(I;L_x^{\tilde{r}_1'} L_y^{ \tilde{r}_2'})}\leq C\|u \|_{L_t^{\infty}(I;L_x^2 L_y^{\frac{2n}{n-2}})}^{\alpha} \| u \|_{L_t^{\tilde{q}'}(I;L_x^{r_1}\dot{W}_y^{1,r_2})}
$$
under the condition \eqref{hol}.
Thus, we obtain
\begin{align*}
\|F(u)\|_{L_t^{\tilde{q}'}(I;L_x^{\tilde{r}_1'} W_y^{1, \tilde{r}_2'})} &\leq C\|u\|_{L_t^{\infty}(I; L_x^2 H_y^1) }^{\alpha} \|u\|_{L_t^{\tilde{q}'}(I;L_x^{ r_1} W_y^{1, r_2})}\\
&\leq CT^{\frac{1}{\tilde{q}'}-\frac{1}{q}}\|u\|_{L_t^{\infty}(I;L_x^2 H_y^1) }^{\alpha} \|u\|_{L_t^{q}(I;L_x^{r_1}W_y^{1, r_2})}
\end{align*}
using $\dot{H}^{1} \hookrightarrow L^{\frac{2n}{n-2}}$ and H\"older's inequality in time with $q>\tilde{q}'$, as desired.

We finish the proof with a remark that the condition \eqref{res} follows from the requirements on $(\tilde{q}, \tilde{r}_1, \tilde{r}_2)$.
Indeed, combining \eqref{hol} with $2\leq \tilde{r}_2 \leq \tilde{r}_1\leq \infty$ implies the second condition in \eqref{res},
while the first one in \eqref{res} follows from a combination with $2\leq \tilde{q}\leq \infty$ and \eqref{admi} for $(q,r_1,r_2)$ and $(\tilde{q},\tilde{r}_1,\tilde{r}_2)$.
\end{proof}


\begin{thebibliography}{9}

\bibitem{BL} J. Bergh and J. L\"ofstr\"om, \textit{Interpolation Spaces, An Introduction}, Springer-Verlag, Berlin-New York, 1976.

\bibitem{CG} R. Carles and C. Gallo, \textit{WKB analysis of nonelliptic nonlinear Schr\"odinger equations},
Commun. Contemp. Math. 22 (2020), 1950045, 19 pp.

\bibitem{C} Z. Chen, \textit{von Neumann-Landau equation for wave functions, wave-particle duality and collapses of wave functions}, arXiv:quant-ph/0703204v2, 2007.

\bibitem{CK} M. Christ and A. Kiselev, \textit{Maximal functions associated to filtrations}, J. Funct. Anal. 179 (2001), no. 2, 409-425.

\bibitem{DMPS} B. Dodson, J. L. Marzuola, B. Pausader and D. P. Spirn, \textit{The profile decomposition for the hyperbolic Schr\"odinger equation},
Preprint, arXiv:1708.08014.

\bibitem{DLS} E. Dumas, D. Lannes and J. Szeftel, \textit{Variants of the focusing NLS equation: derivation, justification,
and open problems related to filamentation}, in Laser filamentation,
CRM Ser. Math. Phys., Springer, Cham, 2016, 19-75.


\bibitem{F} D. Foschi, \textit{Inhomogeneous Strichartz estimates}, J. Hyperbolic Differ. Equ. 2 (2005), 1-24.

\bibitem{GV}  J. Ginibre and G. Velo, \textit{The global Cauchy problem for the nonlinear Schr\"{o}dinger equation revisited}, Ann. Inst. H. Poincar\'{e} Anal. Non Lin\'{e}aire 2 (1985), 309-327.

\bibitem{G} L. Grafakos, \textit{Classical Fourier Analysis}, Springer, New York, 2014.

\bibitem{H} L. H\"ormander, \textit{Estimates for translation invariant operators in $L^p$ spaces}, Acta Math. 104 (1960), 93-140.
	
\bibitem{KT} M. Keel and T. Tao, \textit{Endpoint Strichartz estimates}, Amer. J. Math. 120 (1998), 955-980.

\bibitem{KM} S. Klainerman and M. Machedon, \textit{On the uniqueness of solutions to the Gross-Pitaevskii hierarchy}, Commun. Math. Phys. 279 (2008), 169-185.

\bibitem{L} L. D. Landau, \textit{Das D\"{a}mpfungsproblem in der Wellenmechanik}, Z. Physik, 45 (1927), 430-441.

\bibitem{L2} L. D. Landau and E. M. Lifshitz, \textit{Quantum Mechanics: Non-relativistic Theory}, Pergamon Press, Oxford, 1987.

\bibitem{LL} C. Liu and M. Liu, \textit{On the Cauchy Problem for Von Neumann-Landau Wave Equation}, Journal of Applied Mathematics and Physics, 2 (2014), 1224-1332.

\bibitem{M-S} S. J. Montgomery-Smith, \textit{Time decay for the bounded mean oscillation of solutions of the Schr\"{o}dinger and wave equations}, Duke Math. J. 91 (1998), 393-408.

\bibitem{N} J. von Neumann, \textit{Wahrscheinlichkeitstheoretischer Aufbau der Quantenmechanik},
Nachrichten von der Gesellschaft der Wissenschaften zu G\"ottingen, Mathematisch-Physikalische Klasse 11 (1927), 245-272.

\bibitem{N2} J. von Neumann, \textit{Mathematical Foundations of Quantum Mechanics}, Princeton University Press, Princeton, 1955.


\bibitem{S} E. M. Stein, \textit{Singular integrals and differentiability properties of functions}, Princeton University Press, Princeton, New Jersey, 1970.

\bibitem{St} R. S. Strichartz, \textit{Restrictions of Fourier transforms to quadratic surfaces and decay of solutions of wave equations}, Duke Math. J. 44 (1977), 705-714.

\bibitem{SS} C. Sulem and P.-L. Sulem, \textit{The nonlinear Schr\"odinger equation. Self-focusing and wave collapse},
Applied Mathematical Sciences, 139. Springer-Verlag, New York, 1999.

\bibitem{T} N. Totz, \textit{A justification of the modulation approximation to the 3D full water wave problem},
Comm. Math. Phys. 335 (2015), 369-443.

\bibitem{TW} N. Totz and S. Wu, \textit{A rigorous justification of the modulation approximation to the 2D full water wave problem}.
Comm. Math. Phys. 310 (2012), 817-883.

\end{thebibliography}
\end{document}